\documentclass[a4,reqno,12pt]{amsart}
\usepackage{amsmath,amsthm,amssymb}

\newcommand{\R}{\mathbb{R}}
\newcommand{\C}{\mathbb{C}}
\newcommand{\ru}{\rho} 
\newcommand{\n}[1]{\left\vert {#1} \right\vert}                           
\newcommand{\del}{Q}                   
\newcommand{\cent}{t}               
\newcommand{\Bmet}[1]{ds_{{#1}}^2 }
\newcommand{\hj}{\theta_{n_j}}                 
\newcommand{\dimD}{d} 
\newcommand{\UDL}[1]{\underline{ {#1} }}                   
\newcommand{\Siegel}{D}

\title[Bergman kernel of minimal bounded homogeneous domain]{Some estimates of the Bergman kernel of minimal bounded homogeneous domains}
\author[H. Ishi]{Hideyuki Ishi}



\address{%
Hideyuki Ishi \endgraf
Graduate School of Mathematics \endgraf
Nagoya University \endgraf
Chikusa-ku, Nagoya, 464-8602 \endgraf
Japan
}

\email{hideyuki@math.nagoya-u.ac.jp}

\author[S. Yamaji]{Satoshi Yamaji}
\address{%
Satoshi Yamaji \endgraf
Graduate School of Mathematics \endgraf
Nagoya University \endgraf
Chikusa-ku, Nagoya, 464-8602 \endgraf
Japan
}

\email{satoshi.yamaji@math.nagoya-u.ac.jp}

\date{}
\begin{document}
\newtheorem{Def}{Definition}[section]
\newtheorem{Thm}[Def]{Theorem}
\newtheorem{IntroThm}{Theorem}
\renewcommand{\theIntroThm}{\Alph{IntroThm}}
\newtheorem{Lem}[Def]{Lemma}
\newtheorem{Exer}[Def]{Exercise}
\newtheorem{Prop}[Def]{Proposition}
\newtheorem{Cor}[Def]{Corollary}
\newtheorem{Ex}[Def]{Example}

\makeatletter                
 \renewcommand{\theequation}{%
   \thesection.\arabic{equation}}
   \@addtoreset{equation}{section}
\makeatother


\maketitle

\begin{abstract}
We describe the Bergman kernel of any bounded homogeneous domain in a minimal realization
 relating to the Bergman kernels of the Siegel disks. 
Taking advantage of this expression, 
 we obtain substantial estimates of the Bergman kernel of the homogeneous domain. 
\end{abstract}

\section{Introduction}

In this paper, 
 we discuss the Bergman kernel $K_{\mathcal{U}}$ of a bounded homogeneous domain $\mathcal{U}$,
 which we may assume to be minimal.
One of our main results in the present work is the following estimate of $K_{\mathcal{U}}$,
 which will play a key role
 to characterize the boundedness of the Toeplitz operators in \cite{Yamaji} (see also \cite{Zhu1}). 
\begin{IntroThm} \ \label{MainEst}
Take any $\ru>0$. Then, there exists $C_{\ru}>0$ such that 
$$  C_{\ru}^{-1} \leq \n{\frac{K_{\mathcal{U}}(z,a)}{K_{\mathcal{U}}(a,a)} } \leq C_{\ru} $$
for all $z,a \in \mathcal{U}$ with $\beta_{\mathcal{U}} (z,a) \leq \ru$, 
where $\beta_{\mathcal{U}}$ means the Bergman distance on $\mathcal{U}$. 
\end{IntroThm}

In the case that $\mathcal{U}$ is the Harish-Chandra realization of a bounded symmetric domain,
 Theorem $\ref{MainEst}$ is easily verified from properties of the Bergman kernel
 (see Section 6).
However, for a general bounded homogeneous domain,
 the estimate does not seem to be trivial. 

Our idea for the proof of Theorem \ref{MainEst}
is to introduce certain
 equivariant holomorphic maps $\hj : \mathcal{U} \longrightarrow \mathcal{U}_{n_j}$ for $j=1,...,r (:= \mathrm{rank}\, \mathcal{U})$
 from $\mathcal{U}$ into the Siegel disk $\mathcal{U}_{n_j}$ of rank $n_j$. 
Inspired by Xu \cite{Xu}, we obtain the following formula for the description of $K_{\mathcal{U}}$.

\begin{IntroThm}[Theorem \ref{MainResult1}]\ \label{MainResult01} 
There exist
 integers $s_1,...,s_r$ 
 such that 
\begin{eqnarray*}
 K_{\mathcal{U}}(z,w) 
= \mathrm{Vol}(\mathcal{U})^{-1} \prod^r_{j=1} \left\{ \det \left( I_{n_j} - \hj (z) \overline{\hj (w)} \right) \right\}^{-s_j} 
\end{eqnarray*}
for $z,w \in \mathcal{U}$. 
\end{IntroThm}

Recall that the Bergman kernel $K_{\mathcal{U}_{m}}$ of the Siegel disk $\mathcal{U}_{m}$ is given by 
\begin{eqnarray*}
 K_{\mathcal{U}_m}(Z,W) 
= {\rm Vol} \, (\mathcal{U}_m)^{-1} \det \left( I_{m} - Z \overline{W} \right)^{-(m+1)}.
\end{eqnarray*}
Thus we obtain
$$ K_{\mathcal{U}}(z, w) =
 C \prod^r_{j=1} K_{\mathcal{U}_{n_j}}(\hj(z),\hj(w))^{\frac{s_j}{n_j+1}},
$$  
 which implies that the estimate in Theorem \ref{MainEst} for $\mathcal{U}$
 is reduced to the ones for the symmetric domains $\mathcal{U}_{n_j}$.

Let us explain the organization of this paper. 
In section 2, we review properties of the minimal domains, the Siegel upper half plane and homogeneous Siegel domains. 
In particular, we present in section 2.3 the matrix realization of any homogeneous Siegel domain introduced by the first author \cite{Ishi2006}. 
Based on this realization, we observe a relation between the Bergman distances 
on a homogeneous Siegel domain and the Siegel upper half planes (section 3),
 and 
 introduce minor functions on a homogeneous cone in matrix realization, 
 which coincide with the generalized power functions in Gindikin \cite{Gin} (section 4). 
In section 5, we describe the Bergman kernel of minimal bounded homogeneous domains. 
Since the Bergman kernel of a minimal bounded homogeneous domain is expressed
as a ratio of the Bergman kernels of the corresponding Siegel domain (Lemma $\ref{HyoukaLem1}$), 
we obtain Theorem \ref{MainResult01}. 
In section 6, we prove Theorem \ref{MainEst},
 which yields another important estimate of $K_{\mathcal{U}}$ (Proposition~\ref{prop:KU-Mrho}). \\

\noindent
\textbf{Notation.} \ 
For an $N \times N$ matrix $A = (a_{ij})\in \mathrm{Mat}(N, \C)$ 
 and $k = 1, \dots, N$,
 we write $A^{[k]}$ for the $k \times k$ matrix $(a_{ij})_{1 \le i,j \le k}$. 
For real or complex domain $D$, we denote by $\mathrm{Cl}(D)$ the closure of $D$. 
The complexification of a real vector space $V$ will 
be denoted by $V_{\C}$.
\section{Preliminaries}  
\subsection{Minimal domain}
First of all, we recall the definition and properties of minimal domains (see \cite{I-K}, \cite{MM}). 
Let $D$ be a complex domain in $\C^n$ with finite volume 
and $\cent \in D$. 
We say that $D$ is a minimal domain with a center $\cent$ if the following condition is satisfied: 
for every biholomorphism $\psi : D \longrightarrow D^{\prime}$ with $\det J(\psi , \cent)=1$, we have 
\begin{eqnarray*}
{\rm Vol} \, (D^{\prime}) \geq {\rm Vol} \, (D).
\end{eqnarray*}
We have the following convenient criterion for a domain to be minimal 
(see \cite[Proposition 3.6]{I-K}, \cite[Theorem 3.1]{MM}). 

\begin{Prop} \ \label{I-K3.8} 
Let $D \subset \C^n$ be a bounded univalent domain and $t \in D$. 
Then, $D$ is a minimal domain with a center $\cent$ if and only if 
$$K_{D}(z,\cent) = \frac{1}{{\rm Vol \, }(D)}$$
for any $z \in D$. 
\end{Prop}

For example, a circular domain is minimal with a center $0$,
 so that the Harish-Chandra realization for a bounded symmetric domain is also minimal,
 while there are many other minimal realizations for the symmetric domain.  
Recently in \cite{I-K}, a representative domain turns out to be a nice bounded realization 
 of a bounded homogeneous domain,
 which is a generalization of the Harish-Chandra realization.
The representative bounded homogeneous domain is always 
 a minimal domain with a center $0$ (see \cite[Proposition 3.8]{I-K}),
 though it is not circular unless it is symmetric.  
Therefore, in conclusion,
 every bounded homogeneous domain is biholomorphic to a minimal bounded homogeneous domain. 

\subsection{Siegel upper half plane}
Here we present basic facts used in this paper 
 about the Siegel upper half plane $\mathcal{D}_n$.
It is well known that 
 the real symplectic group $Sp(2n, \R)$ acts on $\mathcal{D}_n$ transitively 
 as linear fractional transforms:
$$
 \alpha \cdot Z = (PZ + Q)(RZ + S)^{-1} 
\qquad \left(
\alpha = \begin{pmatrix} P & Q \\ R & S \end{pmatrix} \in Sp(2n, \R),\,\, Z \in \mathcal{D}_n
\right).
$$
Let $H_n$ be the group of $n \times n$ lower triangular matrices with positive diagonals.
We define 
\begin{eqnarray*}
B_n :=  \left\{ \left(
\begin{matrix}
T & 0 \\
 0  & {^t}T^{-1}
\end{matrix}
\right) \left(
\begin{matrix}
I & X \\
0 & I \\
\end{matrix}
\right) \bigg| 
 X \in {\rm Sym}(n,\R), 
 T \in H_n
\right\} .
\end{eqnarray*}
Then $B_n$ is a maximal connected split solvable Lie subgroup of $Sp(2n, \R)$.
The action of $B_n$ on $\mathcal{D}_n$ is described as
\begin{align} \label{eqn:B-action}
 \beta \cdot Z = T\,Z\,{^t}T + X
 \qquad \left( \beta = \left(
\begin{matrix}
T & 0 \\
 0  & {^t}T^{-1}
\end{matrix}
\right) \left(
\begin{matrix}
I & X \\
0 & I \\
\end{matrix}
\right) \in B_n,\,\,Z \in \mathcal{D}_n
\right),
\end{align}
so that the group $B_n$ acts on $\mathcal{D}_n$ simply transitively.

Let $\mathcal{C}_n$ be the Cayley transform from $\mathcal{D}_n$
 onto the Siegel disk $\mathcal{U}_n$ defined by 
\begin{eqnarray}
 \mathcal{C}_n (Z) := (Z-iI_n)(Z+iI_n)^{-1}
\end{eqnarray}
for $Z \in \mathcal{D}_n$. 
It is easy to see that
\begin{eqnarray}
I_n - \mathcal{C}_n (Z) \ \overline{\mathcal{C}_n (Z^{\prime})}
 = \left( \frac{Z +iI_n}{2i} \right)^{-1} 
 \left( \frac{Z - \overline{Z^{\prime}}}{2i} \right) 
 \overline{\left( \frac{Z^{\prime} + iI_n  }{2i} \right)^{-1} }. \label{Cayleyeq}
\end{eqnarray}

\subsection{Homogeneous Siegel domain}
Let $\Omega$ be a regular open convex cone in a real vector space $V$,
 $W$ a complex vector space, and $F : W \times W \to V_{\C}$ a Hermitian map
 such that
 $F(u,u) \in \mathrm{Cl}(\Omega) \setminus \{0\}$ 
 for $u \in W \setminus \{0\}$.
Then the Siegel domain $\Siegel(\Omega, F) \subset V_{\C} \times W$ is defined by
 $$
 \Siegel(\Omega, F) := \{\, (z,u) \in V_{\C} \times W \,|\,\mathrm{Im}\, z - F(u,u) \in \Omega \,\}.
 $$
For the degenerate case $F=0$ with $W = \{0\}$,
 the Siegel domain becomes a tube domain
 $\Siegel(\Omega) = V + i \Omega \subset V_{\C}$.
It is known that
 every bounded homogeneous domain is biholomorphic to some homogeneous Siegel domain
 (\cite{VGS}).
On the other hand, 
 it is shown in \cite{Ishi2006} that every homogeneous Siegel domain is realized 
 as a set of complex matrices with specific block decompositions in the following way.

Let $\nu_1, \dots, \nu_r$ be positive numbers,
 and $\{\mathcal{V}_{lk}\}_{1 \le k < l \le r}$ a system of real vector spaces
 $\mathcal{V}_{lk} \subset {\rm Mat}(\nu_l,\nu_k;\R)$ satisfying
\begin{eqnarray*}
({\rm V1}) && A \in \mathcal{V}_{lk}, \, B \in \mathcal{V}_{ki}  \Longrightarrow AB \in \mathcal{V}_{li} {\rm \ for \ } 1 \leq i < k < l \leq r, \\
({\rm V2}) && A \in \mathcal{V}_{li}, \, B \in \mathcal{V}_{ki}  \Longrightarrow A \,^t B \in \mathcal{V}_{lk} {\rm \ for \ } 1 \leq i < k < l \leq r, \\
({\rm V3}) && A \in \mathcal{V}_{lk}  \Longrightarrow A \,^t A \in \R I_{\nu_l} {\rm \ for \ } 1 \leq k < l \leq r.
\end{eqnarray*}
We set $\nu := \nu_1 + \dots + \nu_r$.
Let
 $\mathcal{V} \subset \mathrm{Sym}(\nu, \R)$ be the space of real symmetric matrices $X$ of the form
$$
 \begin{pmatrix}  
 X_{11} & {^t}X_{21} & \dots & {^t} X_{r1} \\
 X_{21} & X_{22}    &        & {^t} X_{r2} \\
 \vdots &  & \ddots & & \\
 X_{r1} & X_{r2} & & X_{rr} 
\end{pmatrix}
\quad \left( 
 \begin{gathered}
 X_{kk} = x_{kk} I_{\nu_k},\,\,x_{kk} \in \R \quad (k=1, \dots, r)\\
 X_{lk} \in \mathcal{V}_{lk} \quad (1 \le k < l \le r)
 \end{gathered} 
 \right).
$$
We define $\Omega_{\mathcal{V}} := \{\, X \in \mathcal{V}\,|\,X$ is positive definite$\}$.
Then $\Omega_{\mathcal{V}}$ is a regular open convex cone in the vector space $\mathcal{V}$.
Let $\nu_0$ be a positive integer, and 
 $\{\mathcal{W}_k\}_{1 \le k \le r}$ a system of complex vector spaces
 $\mathcal{W}_k \subset {\rm Mat \,}(\nu_k, \nu_0;\C)$ 
 satisfying
\begin{eqnarray*}
({\rm W1}) && A \in \mathcal{V}_{lk}, \, C \in \mathcal{W}_{k}  
 \Longrightarrow AC \in \mathcal{W}_{l} {\rm \ for \ } 1 \leq  k < l \leq r, \\
({\rm W2}) && C \in \mathcal{W}_{l},  \, C^{\prime} \in \mathcal{W}_{k}  \Longrightarrow C \,^t \overline{C^{\prime}} \in (\mathcal{V}_{lk})_{\C} {\rm \ for \ } 1 \leq i < l \leq r, \\
({\rm W3}) && C \in \mathcal{W}_{k}  \Longrightarrow C \,^t \overline{C} + \overline{C} \,^t C \in 
\R I_{\nu_k} {\rm \ for \ } 1 \leq  k \leq r. 
\end{eqnarray*}
Let $\mathcal{W}$ be the space of complex matrices $U$ of the form
$$
 U = \begin{pmatrix} U_1 \\ U_2 \\ \vdots \\ U_r \end{pmatrix}
 \in \mathrm{Mat}(\nu,\nu_0;\C)
 \quad 
 \Bigl( U_k \in \mathcal{W}_k,\,\,\,k=1, \dots, r \Bigr).
$$
For $U,U' \in \mathcal{W}$, 
 we define
 $F_{\mathcal{V}, \mathcal{W}}(U,U') := (U {^t}\overline{U'} + \overline{U'}{^t}U)/4$.
We see from (W1) -- (W3) that $F_{\mathcal{V}, \mathcal{W}}$ is a $\mathcal{V}_{\C}$-valued Hermitian form.
Furthermore, 
 it is easy to see that
 $F_{\mathcal{V}, \mathcal{W}}(U,U) \in \mathrm{Cl}(\Omega_{\mathcal{V}}) \setminus \{0\}$ 
 for $U \in \mathcal{W} \setminus \{0\}$. 
Then 
 the Siegel domain
 $\Siegel(\Omega_{\mathcal{V}}, F_{\mathcal{V}, \mathcal{W}})$
 is defined by
 $$
 \Siegel(\Omega_{\mathcal{V}}, F_{\mathcal{V}, \mathcal{W}}) 
 := \{ (Z,U) \in \mathcal{V}_{\C} \times \mathcal{W} \,|\, 
      \mathrm{Im}\, Z - F_{\mathcal{V}, \mathcal{W}}(U,U) \in  \Omega_{\mathcal{V}} \,\},
 $$
 which we shall see to be homogeneous.
First, 
 let $H$ be the set of $\nu \times \nu$ lower triangular matrices $T$ of the form
$$
 \begin{pmatrix}
 T_{11} &  &  & \\
 T_{21} & T_{22}    &   & \\
 \vdots &  & \ddots & & \\
 T_{r1} & T_{r2} & & T_{rr} 
\end{pmatrix}
\quad \left( 
 \begin{gathered}
 T_{kk} = t_{kk} I_{\nu_k},\,\,t_{kk}>0 \quad (k=1, \dots, r) \\
 T_{lk} \in \mathcal{V}_{lk} \quad(1 \le k < l \le r)
 \end{gathered} 
 \right).
$$
Then $H$ is a subgroup of the solvable group $H_{\nu}$ thanks to (V1).
Moreover,
 $H$ acts on the cone $\Omega_{\mathcal{V}}$ simply transitively 
 by $\Omega_{\mathcal{V}} \owns X \mapsto T X {^t}T \in \Omega_{\mathcal{V}}\,\,\,(T \in H)$.
For $X \in \mathcal{V}, U \in \mathcal{W}$ and $T \in H$, 
we define an affine transform $b(X,U,T)$ on $\mathcal{V}_{\C} \times \mathcal{W}$ by 
\begin{eqnarray*} 
b(X,U,T) \cdot \zeta^{\prime} 
 &:=& (T \, Z^{\prime} \, {^t}T +X+2i \, F_{\mathcal{V}, \mathcal{W}}(T U^{\prime},U)+i F_{\mathcal{V}, \mathcal{W}}(U,U), \ TU^{\prime}+U)  \\
 & & \hspace{3em} (\zeta^{\prime} =(Z^{\prime},U^{\prime}) \in \mathcal{V}_{\C} \times \mathcal{W})  .
\end{eqnarray*}
Then, each $b(X,U,T)$ preserves the domain $\Siegel(\Omega_{\mathcal{V}}, F_{\mathcal{V}, \mathcal{W}})$.
Let $B$ be the set 
$ \{ b(X,U,T) \mid  X \in \mathcal{V}, U \in \mathcal{W}, T \in H \} $, 
which forms a split solvable Lie group acting on $\Siegel(\Omega_{\mathcal{V}}, F_{\mathcal{V}, \mathcal{W}})$
 simply transitively.
In particular,
 the Siegel domain $\Siegel(\Omega_{\mathcal{V}}, F_{\mathcal{V}, \mathcal{W}})$ is homogeneous.
Since every homogeneous Siegel domain can be obtained this way,
 we shall consider only the Siegel domains of this form.
In particular, for the treatment of a Siegel domain of tube type, 
 we set
 $\mathcal{W}_k = \{0\} \subset \mathrm{Mat}(\nu_k,\nu_0;\C) \,\,\,(k=1, \dots, r)$.
Thus we write $F$ and $\Siegel$ 
 for $F_{\mathcal{V}, \mathcal{W}}$ and $\Siegel(\Omega_{\mathcal{V}}, F_{\mathcal{V}, \mathcal{W}})$
 respectively in what follows for simplicity.

The Siegel domain $\Siegel$ is embedded into the Siegel disk $\mathcal{D}_N$
 ($N := \nu_0 + \nu_1 + \dots + \nu_r = \nu_0 + \nu$) equivariantly with respect to the action of $B$.
Namely,
 if we define an injective holomorphic map $\Phi : \Siegel \to \mathcal{D}_N$ and
 a group homomorphism $\phi : B \to B_N$ by   
\begin{equation}
\Phi(\zeta) := \left(
\begin{matrix}
 i I_{\nu_0}  &  \,^{t}U \\
    U    &  Z-\frac{i}{2} U  \,^{t}U
\end{matrix}
\right) 
\quad (\zeta = (Z,U) \in \Siegel), 
\label{MapPhi}
\end{equation}
 and 
 \begin{align*} 
{} & \phi (b(X,U,T)) \\
& :=     \left(
\begin{matrix}
I_{\nu_0} &         &                   &  {\rm Re \, } {^t}U  \\ 
          & I_{\nu} & {\rm Re \, } U  & X + \frac{1}{2} {\rm Im \, } U \, {^t}U \\
          &         &    I_{\nu_0}      &      \\
          &         &                   & I_{\nu}  \\
\end{matrix}
\right)
\left(
\begin{matrix}
       I_{\nu_0} &      &           &    \\ 
{\rm Im \, } U   & T  &           & \\
                 &      & I_{\nu_0} & - {\rm Im \, } {^t}U {^t} T^{-1}   \\
                 &      &           & {^t} T^{-1}  \\
\end{matrix}
\right) 
\end{align*}
 respectively,
 we have by \cite[{\rm P.}601]{Ishi2006}
 \begin{eqnarray}
 \phi (b) \cdot \Phi(\zeta) = \Phi(b \cdot \zeta)  \label{Sayou1}
 \end{eqnarray}
 for any $b \in B$ and $\zeta \in \Siegel$.

\section{Equivariant maps into the Siegel upper half planes}
For $n=1, \dots, N$,
 let $\pi_n : \mathcal{D}_N \to \mathcal{D}_n$ be the surjective holomorphic map
 given by $\pi_n(Z) := Z^{[n]}\,\,\,(Z \in \mathcal{D}_N)$.
Let us observe the equivariance of $\pi_n$
 under actions of solvable groups.  
We define $\rho_n : B_N \longrightarrow B_n$ by 
\begin{eqnarray*}
\rho_n \left(  \left(
\begin{matrix}
T & 0 \\
0 & {^t}T^{-1}
\end{matrix}
\right) \left(
\begin{matrix}
I & X \\
0 & I \\
\end{matrix}
\right)  \right) =  \left(
\begin{matrix}
T^{[n]} & 0 \\
0 & {{^t}T^{[n]}}^{-1}
\end{matrix}
\right) \left(
\begin{matrix}
I & X^{[n]} \\
0 & I \\
\end{matrix}
\right).
\end{eqnarray*}
Then $\rho_n$ is a group homomorphism, and we see from (\ref{eqn:B-action}) that
\begin{eqnarray}
\pi_n(\beta \cdot Z) = \rho_n(\beta) \cdot \pi_n(Z) \label{Sayou2}
\end{eqnarray}
for any $Z \in \mathcal{D}_N$ and $\beta \in B_N$. 
Now we define $\Phi_n := \pi_n \circ \Phi$ and $\phi_n :=\rho_n \circ \phi$. 
From $(\ref{Sayou1})$ and $(\ref{Sayou2})$, we obtain the following proposition. 

\begin{Prop}\ \label{MapKakan}
One has  
\begin{eqnarray}
\phi_n(b) \cdot  \Phi_n(\zeta) = \Phi_n ( b \cdot \zeta)  \label{MapKakan1}
\end{eqnarray}
for all $b \in B$ and $\zeta \in \Siegel$.
\end{Prop}

Using the group equivariance of $\Phi_n$,
 we shall show that the map $\Phi_n$ is Lipschitz continuous with respect to the Bergman distances
 on $\Siegel$ and $\mathcal{D}_n$.

\begin{Prop}\  \label{BergmanMet}
There exists $M_n >0$ such that 
\begin{eqnarray*}
 \beta_{\mathcal{D}_n}(\Phi_n(\zeta), \Phi_n(\eta))  \leq  M_n \, \beta_{\Siegel}(\zeta,\eta) 
\end{eqnarray*}
for any $\zeta,\,\eta \in \Siegel$. 
\end{Prop}

\begin{proof}
Let $\Bmet{\mathcal{D}_n}$ (resp. $\Bmet{\Siegel}$) be the Bergman metric on $\mathcal{D}_n$ (resp. $\Siegel$). 
It suffices to prove 
\begin{eqnarray} 
 \Bmet{\mathcal{D}_n} (\Phi_n(\zeta);J(\Phi_n,\zeta)X) \leq M_n \, \Bmet{\Siegel} (\zeta; X) \label{Berg.Metric1}
\end{eqnarray}
for all $\zeta \in \Siegel$ and $X \in \C^{\dimD}$, 
 where $\dimD := \dim \Siegel$. 

Since Hermitian forms $(\Bmet{\Siegel} )_{p_0}$
 and $(\Phi_m^{\ast}\Bmet{\mathcal{D}_m} )_{p_0}$ 
 on $\C^{\dimD}$
 are
 positive definite and positive semi-definite respectively, 
 we can take $M_n$ for which
\begin{eqnarray} 
(\Phi_n^{\ast}\Bmet{\mathcal{D}_n} )_{p_0} \leq M_n (\Bmet{\Siegel} )_{p_0} .  \label{Itten}
\end{eqnarray}
Using homogeneity, we will prove $(\ref{Berg.Metric1})$ holds for all $\zeta \in \Siegel$ and $X \in \C^{\dimD}$. 
Let us take $b \in B$ such that $ b \cdot \zeta = p_0$.
Then, the right hand side of $(\ref{Berg.Metric1})$ is written as 
\begin{eqnarray} 
  \Bmet{\mathcal{D}_n} (\Phi_n(b \cdot p_0)\,;J(\Phi_n, b \cdot p_0)X)  . \label{phiD} 
\end{eqnarray}
Substituting $\zeta = p_0$ in $(\ref{MapKakan1})$, we have 
\begin{eqnarray*}
 \Phi_n ( b \cdot p_0)  = \phi_n(b) \cdot  \Phi_n (p_0)  .
\end{eqnarray*}
Furthermore, differentiating $(\ref{MapKakan1})$ at $\zeta = p_0$, we obtain 
\begin{eqnarray*}
J( \Phi_n , b \cdot p_0)  \, J (b , p_0) 
 =   J( \phi_n(b) , \Phi_m (p_0)) \, \, J (\Phi_n , p_0) . 
\end{eqnarray*}
Therefore, $(\ref{phiD})$ is equal to 
\begin{eqnarray*} 
   \Bmet{\mathcal{D}_n} \left( \phi_n(b) \cdot  \Phi_n (p_0)  \,; J( \phi_n(b) , \Phi_n (p_0)) \,  J (\Phi_n , p_0) J (b , p_0)^{-1} X \right)  ,
\end{eqnarray*}
which is equal to 
\begin{eqnarray*} 
 \Bmet{\mathcal{D}_n} \left(  \Phi_n (p_0)  \,;J (\Phi_n , p_0) J (b , p_0)^{-1} X \right) 
\end{eqnarray*}
because $\Bmet{\mathcal{D}_n}$ is invariant under the holomorphic automorphism $\phi_n(b)$. 
By $(\ref{Itten})$,  
\begin{eqnarray*} 
 \Bmet{\mathcal{D}_n} \left(  \Phi_n (p_0)  \,;J (\Phi_n , p_0) J (b , p_0)^{-1} X \right) 
     \leq  M_n \Bmet{\Siegel} \left( p_0  \,;J (b , p_0)^{-1} X \right) 
\end{eqnarray*}
and the right hand side is equal to 
\begin{eqnarray*} 
   M_n \Bmet{\Siegel} \left( b \cdot  p_0  \,;X \right)   =   M_n \Bmet{\Siegel} \left(\zeta  \,;X \right) , 
\end{eqnarray*} 
because $b \in B$ is a biholomorphic map on $\Siegel$. 
Therefore, $(\ref{Berg.Metric1})$ is verified, whence Proposition $\ref{BergmanMet}$ follows. 
\end{proof}

\section{Minor functions}

\begin{Def} 
We set $\mu_1 := 1$ and $\mu_j := \nu_1 + \nu_2 + \cdots + \nu_{j-1} +1$ for $2 \leq j \leq r$. 
For $Z \in \Omega_{\mathcal{V}} + i \mathcal{V}\, (\subset \mathcal{V}_{\C})$,
 we define $\del_j(Z)\,\,\,(j=1, \dots, r)$ by
\begin{eqnarray*}
 \del_j(Z) := \frac{\det Z^{[\mu_{j}]} }{\det Z^{[\mu_j -1]}},
\end{eqnarray*}
 where we interpret $\det Z^{[0]} = 1$. 
We also define $Q^{\UDL{s}} (Z)$ 
 for $\UDL{s} := (s_1,...,s_r) \in \R^{r}$ by
\begin{eqnarray*}
  Q^{\UDL{s}} (Z) :=  \del_1 (Z)^{s_1} \cdots  \del_r (Z)^{s_r}.
\end{eqnarray*}
\end{Def}

The functions $\del_j(Z)$ and $Q^{\UDL{s}} (Z)$ are denoted by $\chi_j(Z)$ and $Z^s$ respectively in \cite{Gin}. 
If $\Siegel$ is a symmetric Siegel domain, then $Q^{\UDL{s}}$ coincides with 
the generalized power function $\Delta_{\UDL{s}}$ in \cite[P.122]{F-K}. 

\begin{Ex}
Let $\mathcal{V}$ be the set of $4 \times 4$ symmetric matrices with real entries of the form 
\begin{eqnarray*}
X =  \left(
\begin{matrix}
x_1  &  0   &  x_4 &  0   \\
 0    & x_1 &   0   & x_5 \\
x_4  &  0   &  x_2 &  0   \\
  0   & x_5 &   0   & x_3 \\
\end{matrix}
\right).
\end{eqnarray*}
In this case,
 $\nu_1 = 2,\,\nu_2 = \nu_3 = 1,\,
 \mathcal{V}_{21} = \{\begin{pmatrix} x_4 & 0 \end{pmatrix} \,| \,x_4 \in \R \},
 \mathcal{V}_{31} = \{\begin{pmatrix} 0 & x_5 \end{pmatrix} \,| \,x_5 \in \R \}$
 and $\mathcal{V}_{32} = \{0\}$. 
The cone $\Omega_{\mathcal{V}}$ is nothing but the Vinberg cone
 (\cite{B-N}, \cite{Vin}).
Let $\mathcal{W} = \{ 0 \}$ and $F = 0$. 
Then $\Siegel$ is the tube domain $\mathcal{V} + i \Omega_{\mathcal{V}}$
 and
 we obtain 
\begin{eqnarray*}
\del_1(X) &=&   x_1  , \\
\del_2(X) &=&  \frac{x_1(x_1 x_2 - x_4^2)}{x_1^2} =  x_2 - \frac{x_4^2}{x_1}  ,  \\
\del_3(X) &=&  \frac{\det X}{x_1(x_1 x_2 - x_4^2)} = x_3 - \frac{x_5^2}{x_1}  .
\end{eqnarray*}
These functions are considered in \cite{B-N}. 
\end{Ex}

Take any $X \in \mathcal{V}$. 
There exists a unique lower triangular matrix $T \in H$ such that $T \,^tT = X$. 
Then we have
\begin{eqnarray}
\det X^{[m]} 
  = \det (T \,^tT)^{[m]} = (\det T^{[m]})^2   \label{det.}
\end{eqnarray}
for any $1 \leq m \leq \nu$. 
Since $T$ is a lower triangular matrix, we can easily calculate the right hand side of $(\ref{det.})$. We have 
\begin{eqnarray*}
\det X^{[\mu_k]} 
  &=& {t_{11}}^{2\nu_1} {t_{22}}^{2\nu_2} \cdots {t_{k-1,k-1}}^{2\nu_{k-1}} t_{kk}^2 , \\
\det X^{[\mu_k-1]} 
  &=& {t_{11}}^{2\nu_1} {t_{22}}^{2\nu_2} \cdots {t_{k-1,k-1}}^{2\nu_{k-1}}   .  \label{DetNoHyouji1}
\end{eqnarray*}
Therefore, we obtain $\del_k(X)    =  {t_{kk}}^2 $. Hence we have 
\begin{eqnarray}
\det X^{[\mu_k]} =  {\del_1(X)}^{\nu_1} {\del_2(X)}^{\nu_2} \cdots {\del_{k-1}(X)}^{\nu_{k-1}} {\del_k(X)}.   \label{DetNoHyouji3}
\end{eqnarray}
Let $\det^{[m]}$ be the polynomial function on $\mathcal{V}_{\C}$ defined by
 $\det^{[m]} Z := \det Z^{[m]}\,\,\,(Z \in \mathcal{V}_{\C})$.
\begin{Lem}\ \label{DelToDet--jyunbi}  
Let $2 \leq k \leq r$. Then, one has
\begin{eqnarray}
 \del_k 
= ( {\rm \det}^{[\mu_1]} )^{c_{k1}}  \, ( {\rm \det}^{[\mu_2]} )^{c_{k2}} \cdots ( {\rm \det}^{[\mu_{k-1}]} )^{c_{k,k-1}} \, {\rm \det}^{[\mu_k]} , \label{Cij}
\end{eqnarray}
where $c_{ki}= - \nu_k \prod_{i<p<k} (1-\nu_p)$ for $i=1, \dots, k-1$. 
\end{Lem}
\begin{proof}

For $1 \leq i <r$, we obtain 
\begin{eqnarray*}
\frac{{\rm det}^{[\mu_{i+1}]}}{{\rm det}^{[\mu_i]}} = \frac{{\del_{i}}^{\nu_{i}} \del_{i+1}}{\del_i}
\end{eqnarray*}
from $(\ref{DetNoHyouji3})$. Therefore, we have
\begin{eqnarray}
\del_{i+1} = {\del_i}^{1-\nu_i} \frac{{\rm det}^{[\mu_{i+1}]}}{{\rm det}^{[\mu_i]}}.    \label{Zen1}
\end{eqnarray}
In particular, 
\begin{eqnarray*}
 \del_2 = {\del_1}^{1-\nu_1} \frac{{\rm det}^{[\mu_{2}]}}{{\rm det}^{[\mu_1]}}  
        = ({\rm det}^{[\mu_{1}]})^{-\nu_1} {\rm det}^{[\mu_{2}]}  .
\end{eqnarray*}
Thus, the formula holds for $k=2$ with $c_{21} = - \nu_1$. 
Assume that the statement holds for $l=j$. 
Substituting $(\ref{Cij})$ to $(\ref{Zen1})$, we have
\begin{eqnarray*}
\del_{j+1} &=& {\rm det}^{[\mu_{j+1}]} \left\{ {\rm det}^{[\mu_j]} \prod_{ i <j} ({\rm det}^{[\mu_i]})^{c_{ji}} \right\}^{(1-\nu_j)} ({\rm det}^{[\mu_j]})^{-1}  \\
 &=& {\rm det}^{[\mu_{j+1}]} ({\rm det}^{[\mu_j]})^{-\nu_j} \prod_{ i <j} ({\rm det}^{[\mu_i]})^{(1-\nu_j)c_{ji}} .
\end{eqnarray*}
Therefore, if we put 
\begin{eqnarray*}
c_{j+1,i}    &=& (1-\nu_j)c_{ji} = - \nu_j \prod_{i<p<j+1} (1-\nu_p)   \ \ \ {\rm for \ } i=1, \dots, j ,  \\
c_{j+1,j}  &=& - \nu_j ,
\end{eqnarray*}
we have $(\ref{Cij})$ for $k=j+1$. 
\end{proof}
Let $\UDL{d} := (d_1, \dots, d_r)$ and $\UDL{b} := (b_1, \dots, b_r)$ be $r$-tuples of integers defined by
$$
 d_k := 2 + \sum_{i<k} \dim \mathcal{V}_{ki} + \sum_{l>k} \dim \mathcal{V}_{lk}, \quad
 b_k := \dim \mathcal{W}_k \qquad (k=1, \dots, r).
$$
Then the Bergman kernel of the homogeneous Siegel domain $\Siegel$ is given by
\begin{eqnarray*}
K_{\Siegel}(\zeta,\zeta^{\prime})  = C \, Q^{-(2\UDL{d} + \UDL{b})} \left( \frac{Z-\overline{Z^{\prime}}}{2i} - F(U,U^{\prime}) \right) 
\end{eqnarray*}
for $\zeta := (Z,U), \zeta^{\prime} := (Z^{\prime},U^{\prime}) \in \Siegel$,
 where $C$ is a constant depending only on the normalization of the Lebesgue measure on $\mathcal{V}$ 
 (see \cite[Proposition 5.1]{Gin}). 
This formula together with Lemma $\ref{DelToDet--jyunbi}$ tells us the following.

\begin{Lem}\ \label{DelToDet}
Let $c_{kj}\,\,\,(1 \leq j < k  \leq r)$ be the integer defined as in Lemma $\ref{DelToDet--jyunbi}$, and $c_{jj}=1$ for $1 \leq j \leq r$. 
Setting
\begin{eqnarray*}
 s_j :=  \sum_{k \geq j} (2d_k+b_k) c_{kj} \qquad (j=1, \dots, r),
\end{eqnarray*}
 one has
\begin{eqnarray*}
  K_{\Siegel}(\zeta,\zeta^{\prime}) 
  =  C \, \prod_{j=1}^r   \left\{ {\rm \det}^{[\mu_j]}  \left( \frac{Z-\overline{Z^{\prime}}}{2i} - F(U, U^{\prime}) \right) \right\}^{-s_j}.
\end{eqnarray*}
\end{Lem}

\section{Bergman kernel of the minimal bounded homogeneous domains}
By the transformation formula of the Bergman kernel, we have the following general formula. 
\begin{Lem}\ \label{Domchange} 
Let $D_1$ and $D_2$ be complex domains and $\alpha$ a biholomorphic map from $D_1$ onto $D_2$. 
Then, one has 
\begin{eqnarray}
\frac{K_{D_2} (\alpha(\zeta_1),\alpha(\zeta_2))  \, K_{D_2}(\alpha(\zeta_3),\alpha(\zeta_4))}
{K_{D_2} (\alpha(\zeta_1),\alpha(\zeta_4))  \, K_{D_2}(\alpha(\zeta_3),\alpha(\zeta_2))}
 = \frac{K_{D_1} (\zeta_1,\zeta_2) \, K_{D_1}(\zeta_3,\zeta_4)}
 {K_{D_1} (\zeta_1,\zeta_4)  \, K_{D_1}(\zeta_3,\zeta_2)}
\end{eqnarray}
for $\zeta_1,\zeta_2,\zeta_3,\zeta_4 \in D_1$. 
\end{Lem}

Let $\mathcal{U}$ be a minimal bounded homogeneous domain with a center $\cent$.
By \cite{VGS}, $\mathcal{U}$ is biholomorphic to a homogeneous Siegel domain 
 $D \subset \mathcal{V}_{\C} \times \mathcal{W}$.
We set $p_0 := (i I, 0) \in D$ and take
 a biholomorphic map $\sigma$ from $D$ onto $\mathcal{U}$
 such that $\sigma(p_0) = t$. 
From Lemma $\ref{Domchange}$, we have the following lemma, which is substantial
 in the present work.
\begin{Lem}\ \label{HyoukaLem1} 
For any $\zeta,\zeta^{\prime} \in \Siegel$, one has 
\begin{eqnarray*}
{\rm Vol \, }(\mathcal{U}) \, K_{\mathcal{U}}(\sigma(\zeta), \sigma(\zeta^{\prime})) 
 =  \frac{K_{\Siegel} (\zeta,\zeta^{\prime}) K_{\Siegel}(p_0,p_0) }
 {K_{\Siegel}(\zeta,p_0) K_{\Siegel} (p_0,\zeta^{\prime})}. 
\end{eqnarray*}
\end{Lem}

\begin{proof}
By Lemma $\ref{Domchange}$, we obtain  
\begin{eqnarray}
  \frac{K_{\mathcal{U}}(\sigma(\zeta), \sigma(\zeta^{\prime})) \, K_{\mathcal{U}}(\cent, \cent)}
 { K_{\mathcal{U}}\left(\sigma(\zeta), \cent\right) \, K_{\mathcal{U}}(\cent, \sigma(\zeta^{\prime}))}
= \frac{K_{\Siegel} (\zeta,\zeta^{\prime})  \, K_{\Siegel}(p_0,p_0)}
{K_{\Siegel}(\zeta,p_0) \, K_{\Siegel} (p_0,\zeta^{\prime})}  . \label{BKkankeisiki}
\end{eqnarray}
By Proposition $\ref{I-K3.8}$, we obtain 
\begin{eqnarray*}
K_{\mathcal{U}}\left(\sigma(\zeta), \cent \right) = K_{\mathcal{U}}(\cent, \sigma(\zeta^{\prime})) = K_{\mathcal{U}}(\cent,\cent) = \frac{1}{{\rm Vol \, }(\mathcal{U})} .
\end{eqnarray*}
Therefore, the left hand side of $(\ref{BKkankeisiki})$ is equal to 
\begin{eqnarray*}
{\rm Vol \, }(\mathcal{U}) \, K_{\mathcal{U}}(\sigma(\zeta), \sigma(\zeta^{\prime})) ,
\end{eqnarray*}
and we complete the proof. 
\end{proof}

For $1 \le n \le N$, let $\theta_n$ be the composition map $\mathcal{C}_n \circ \Phi_n \circ \sigma^{-1}$
 from $\mathcal{U}$ into the Siegel disk $\mathcal{U}_n$.
Using the maps $\theta_n$, we can describe the Bergman kernel $K_{\mathcal{U}}$ as follows.

\begin{Thm}\ \label{MainResult1}
Putting $n_j = \nu_0+\nu_1+ \cdots + \nu_{j-1}+1$, one has
\begin{eqnarray*}
 K_{\mathcal{U}}(z, z^{\prime}) 
&=& \frac{1}{{\rm Vol \, }(\mathcal{U})} \prod^r_{j=1} \left\{ \det \left( I_{n_j} - \hj (z) \overline{\hj (z^{\prime})} \right) \right\}^{-s_j} \\
&=& C \prod^r_{j=1} K_{\mathcal{U}_{n_j}}(\hj(z),\hj(z^{\prime}))^{\frac{s_j}{n_j+1}}  
\end{eqnarray*}
for $z , z^{\prime} \in \mathcal{U}$, where $s_j$ are integers defined by Lemma $\ref{DelToDet}$. 
\end{Thm}
\begin{proof}
Let $\zeta =(Z,U)$ and $\zeta^{\prime}=(Z^{\prime},U^{\prime})$ be elements $\sigma^{-1}(z) \in \Siegel$ and  $\sigma^{-1}(z') \in \Siegel$ respectively.
By Lemma $\ref{HyoukaLem1}$, we have 
\begin{align}
{} &  K_{\mathcal{U}}(z,z^{\prime}) 
= K_{\mathcal{U}}(\sigma(\zeta), \sigma(\zeta^{\prime})) \nonumber  \\
&= \frac{1}{{\rm Vol \, }(\mathcal{U})} \frac{K_{\Siegel} (\zeta,\zeta^{\prime}) K_{\Siegel}(p_0,p_0) }
 {K_{\Siegel}(\zeta,p_0) K_{\Siegel} (p_0,\zeta^{\prime})} \nonumber  \\
&= \frac{1}{{\rm Vol \, }(\mathcal{U})} \frac{Q^{-(2\UDL{d}+\UDL{b})} \left( \frac{Z-\overline{Z^{\prime}}}{2i} -F(U,U^{\prime}) \right) \, Q^{-(2\UDL{d}+\UDL{b})} \left( \frac{iI_{\nu}-\overline{iI_{\nu}}}{2i} -F(0,0) \right)}
 {Q^{-(2\UDL{d}+\UDL{b})} \left( \frac{Z-\overline{iI_{\nu}}}{2i} -F(U,0) \right) \, Q^{-(2\UDL{d}+\UDL{b})} \left( \frac{iI_{\nu}-\overline{Z^{\prime}}}{2i} -F(0,U^{\prime}) \right)} \nonumber \\
&= \frac{1}{{\rm Vol \, }(\mathcal{U})} \frac{Q^{-(2\UDL{d}+\UDL{b})} \left( \frac{Z-\overline{Z^{\prime}}}{2i} -F(U,U^{\prime}) \right) }
 {Q^{-(2\UDL{d}+\UDL{b})} \left( \frac{Z-\overline{iI_{\nu}}}{2i}  \right) \, Q^{-(2\UDL{d}+\UDL{b})} \left( \frac{iI_{\nu}-\overline{Z^{\prime}}}{2i}  \right)}  . \label{BergWoDetHe} 
\end{align}
By Lemma $\ref{DelToDet}$, 
the right hand side of $(\ref{BergWoDetHe})$ is equal to 
\begin{eqnarray*}
 \frac{1}{{\rm Vol \, }(\mathcal{U})} \prod^r_{j=1} \left\{ \frac{\det^{[\mu_j]} \left( \frac{Z-\overline{Z^{\prime}}}{2i} -F(U,U^{\prime}) \right) }
      {\det^{[\mu_j]} \left( \frac{Z-\overline{iI_{\nu}}}{2i}  \right) \, \det^{[\mu_j]} \left( \frac{iI_{\nu}-\overline{Z^{\prime}}}{2i}  \right)} \right\}^{-s_j}. 
\end{eqnarray*}
We see from (\ref{MapPhi}) that
\begin{eqnarray*}
{\rm \det} \left( \frac{Z-\overline{Z^{\prime}}}{2i} - F ( U , U^{\prime}) \right) 
   =  \det \left( \frac{\Phi(\zeta)-\overline{\Phi(\zeta^{\prime})}}{2i} \right) .
\end{eqnarray*}
Moreover, it is not difficult to check that 
\begin{eqnarray*}
{\rm \det}^{[m]} \left( \frac{Z-\overline{Z^{\prime}}}{2i} - F ( U , U^{\prime}) \right) 
   =  \det \left( \frac{\Phi_{m+\nu_0}(\zeta)-\overline{\Phi_{m+\nu_0}(\zeta^{\prime})}}{2i} \right) 
\end{eqnarray*}
for $1 \leq m \leq \nu$.
Since $n_j = \nu_0+\mu_j$ , we have
\begin{eqnarray}
&& \frac{\det^{[\mu_j]} \left( \frac{Z-\overline{Z^{\prime}}}{2i} -F(U,U^{\prime}) \right) }
     { \det^{[\mu_j]} \left( \frac{iI_{\nu}-\overline{Z^{\prime}}}{2i} \right) \, \det^{[\mu_j]} \left( \frac{Z-\overline{iI_{\nu}}}{2i} \right) }   \nonumber  \\
&=& \det \left\{ \left( \frac{\Phi_{n_j}(\zeta) - \overline{\Phi_{n_j}(p_0)}}{2i} \right)^{-1} 
\left( \frac{\Phi_{n_j}(\zeta) - \overline{\Phi_{n_j}(\zeta^{\prime})}}{2i} \right) 
\left( \frac{\Phi_{n_j}(p_0) - \overline{\Phi_{n_j}(\zeta^{\prime})}}{2i} \right)^{-1} 
\right\}  \nonumber \\
&=& \det \left\{ \left( \frac{\Phi_{n_j}(\zeta) +iI_{n_j}}{2i} \right)^{-1}  
\left( \frac{\Phi_{n_j}(\zeta) - \overline{\Phi_{n_j}(\zeta^{\prime})}}{2i} \right) 
\overline{\left( \frac{\Phi_{n_j}(\zeta^{\prime}) + iI_{n_j}  }{2i} \right)^{-1} }
\right\}  .  \label{3tunoSeki}
\end{eqnarray}
By $(\ref{Cayleyeq})$, the last term of $(\ref{3tunoSeki})$ is equal to 
\begin{eqnarray*}
\det \left( I_{n_j} - \mathcal{C}_{n_j} ( \Phi_{n_j} (\zeta)) \ \overline{\mathcal{C}_{n_j} (\Phi_{n_j} (\zeta^{\prime}))} \right) .
\end{eqnarray*}
Since $\theta_{m} = \mathcal{C}_m \circ \Phi_m \circ \sigma^{-1}$, we have 
\begin{eqnarray*}
 K_{\mathcal{U}}(z, z^{\prime}) 
= \frac{1}{{\rm Vol \, }(\mathcal{U})}  \prod^r_{j=1} \left\{ \det \left( I_{n_j} - \theta_{n_j} (z) \, \overline{\theta_{n_j} (z^{\prime})} \right) \right\}^{-s_j} .
\end{eqnarray*}
The second equality in the statement follows from the above and
 the formula (\cite{Hua})
\begin{eqnarray*} 
K_{\mathcal{U}_{n_j}}(w,w^{\prime}) = \frac{1}{{\rm Vol}(\mathcal{U}_{n_j})}
 \det \left( I_{n_j} - w \, \overline{w^{\prime}} \right)^{-(n_j+1)} \ \ \ (w,w^{\prime} \in \mathcal{U}_{n_j}).
\end{eqnarray*}
\end{proof}

\section{Estimates of the Bergman kernel of minimal bounded homogeneous domains} 
In this section, we prove Theorem $\ref{MainEst}$. 
Let $\mathcal{U}$ be a minimal bounded homogeneous domain 
 with a center $\cent$
 as in the previous section.
For each $a \in \mathcal{U}$, we take $\varphi_a$ an automorphism on $\mathcal{U}$ such that 
$\varphi_a (a) =\cent$. Since $K_{\mathcal{U}}( \cdot ,\cent)$ is a constant function, we have
\begin{eqnarray}
\frac{K_{\mathcal{U}}(z,a)}{K_{\mathcal{U}}(a,a)}  \nonumber  
   &=& \frac{K_{\mathcal{U}}(z,a) \, K_{\mathcal{U}}(a,\cent)}{K_{\mathcal{U}}(a,a)K_{\mathcal{U}}(z,\cent)}  \nonumber    \\
   &=& \frac{K_{\mathcal{U}}(\varphi_a(z),\cent) \, K_{\mathcal{U}}(\cent,\varphi_a(\cent))}{K_{\mathcal{U}}(\cent,\cent) \, K_{\mathcal{U}}(\varphi_a(z),\varphi_a(\cent))}    \nonumber \\
   &=& \frac{K_{\mathcal{U}}(\cent,\cent)}{K_{\mathcal{U}}(\varphi_a(z), \varphi_a (\cent))} ,  \label{Beq}
\end{eqnarray}
where the second equality follows from Lemma $\ref{Domchange}$. 
For $\rho>0$,
 let $B(t,\rho)$ denote the closed Bergman disk
 $\{z \in \mathcal{U} \,|\, \beta_{\mathcal{U}}(z,t) \le \rho \}$,
 which is a compact subset of $\mathcal{U}$.
For any $z,a \in \mathcal{U}$ with $\beta_{\mathcal{U}} (z,a) \leq \ru$, we have      
$(\varphi_a(z), \varphi_a (\cent)) \in B(\cent, \ru) \times \mathcal{U} $
because 
$   \beta_{\mathcal{U}} ( \varphi_a(z), \cent) =   \beta_{\mathcal{U}} ( \varphi_a(z), \varphi_a(a)) = \beta_{\mathcal{U}} (z,a) \leq \ru $. 
If $\mathcal{U}$ is a bounded symmetric domain, we know that 
\begin{center}
$(P)$ \ \  $K_{\mathcal{U}}(z_1,z_2)$ extends to the compact set $B(\cent, \ru) \times \mathrm{Cl}(\mathcal{U})$ as a continuous and non-zero function 
\end{center}
(see \cite[Theorem 2.10]{OL}). 
Therefore, we obtain Theorem $\ref{MainEst}$ from $(\ref{Beq})$. 
However, we don't know whether a nonsymmetric homogeneous domain has the property (P). 
Therefore, we take advantage of Theorem $\ref{MainResult1}$, 
which describes
the Bergman kernel of the minimal homogeneous domain 
$\mathcal{U}$
in terms of the Bergman kernel of the Siegel disks $\mathcal{U}_{n_j}$. 
Moreover,
 we have an estimate of the Bergman distance in Proposition $\ref{BergmanMet}$. 
Using these results, we will prove our main theorem. 

\noindent
\textbf{Proof of Theorem \ref{MainEst}.} \indent
Take any $z,a \in \mathcal{U}$ with $\beta_{\mathcal{U}} (z,a) \leq \ru$. 
Then, there exist $\zeta, \eta \in \Siegel$ such that $\sigma(\zeta) = z$ and $\sigma(\eta) =a$. 
By Theorem $\ref{MainResult1}$, we have 
\begin{eqnarray} 
\n{\frac{K_{\mathcal{U}}(z,a)}{K_{\mathcal{U}}(a,a)} }
    &=& \n{\frac{K_{\mathcal{U}}(\sigma(\zeta),\sigma(\eta))}{K_{\mathcal{U}}(\sigma(\eta),\sigma(\eta))} }  \nonumber \\
    &=& \prod^r_{j=1} \n{ \frac{K_{\mathcal{U}_{n_j}} (\hj (\zeta), \hj (\eta))}{K_{\mathcal{U}_{n_j}} (\hj (\eta),\hj (\eta))} }^{\frac{s_j}{n_j+1}} . \label{MainResult2-3} 
\end{eqnarray}
On the other hand,
 since $\beta_{\Siegel}(\zeta,\,\eta) = \beta_{\mathcal{U}}(z,a) \le \rho$,
 we obtain 
\begin{eqnarray} 
\beta_{\mathcal{U}_{n_j}}(\hj (\zeta),\hj (\eta))  
\leq M_{n_j} \ru   \label{MainResult2-5} 
\end{eqnarray}
from Proposition $\ref{BergmanMet}$. Since $\mathcal{U}_{n_j}$ is a bounded symmetric domain, there exists a positive constant $C_j$ such that 
\begin{eqnarray*} 
C_j^{-1} \leq \n{ \frac{K_{\mathcal{U}_{n_j}} (w,w^{\prime})}{K_{\mathcal{U}_{n_j}} (w^{\prime},w^{\prime})} } \leq C_j
\end{eqnarray*}
holds for any $w,w^{\prime} \in \mathcal{U}_{n_j}$ with $\beta_{\mathcal{U}_j}(w,w^{\prime})   \leq  M_{n_j} \ru$. 
Therefore, if $\beta_{\mathcal{U}}(z,a) \leq \rho$, we have  
\begin{eqnarray*} 
C_j^{-1} \leq \n{ \frac{K_{\mathcal{U}_{n_j}} (\hj (\zeta), \hj (\eta))}{K_{\mathcal{U}_{n_j}} (\hj (\eta),\hj (\eta))} } \leq C_j
\end{eqnarray*}
thanks to $(\ref{MainResult2-5})$. In view of $(\ref{MainResult2-3})$, we have 
\begin{eqnarray*} 
C_{\ru}^{-1}  \leq \n{\frac{K_{\mathcal{U}}(z,a)}{K_{\mathcal{U}}(a,a)} } \leq C_{\ru}
\end{eqnarray*}
with 
\begin{eqnarray*} 
C_{\ru} =  \prod^r_{j=1}  C_j^{\frac{s_j}{n_j+1}}. \qquad \qed
\end{eqnarray*}

As an application of Theorem \ref{MainEst}, we obtain another important estimate of $K_{\mathcal{U}}$. 
\begin{Prop} \label{prop:KU-Mrho}
For any $\ru >0$, there exists $M_{\ru}>0$ such that $M_{\ru}^{-1} \leq \n{K_{\mathcal{U}}(z,w)} \leq M_{\ru}$
 for any $z \in B(\cent,\ru)$ and $w \in \mathcal{U}$. 
\end{Prop}
\begin{proof}
Similarly to $(\ref{Beq})$, we obtain 
\begin{align}
\n{\frac{K_{\mathcal{U}}(z,w)}{K_{\mathcal{U}}(\cent,\cent)} } 
   = \n{\frac{K_{\mathcal{U}}(z,w) \, K_{\mathcal{U}}(\cent,\cent)}{K_{\mathcal{U}}(z,\cent) \, K_{\mathcal{U}}(\cent,w)} } 
   = \n{\frac{K_{\mathcal{U}}(\varphi_w(\cent),\varphi_w(\cent)) }{K_{\mathcal{U}}(\varphi_w(z),\varphi_w(\cent))} }.  \label{5.4}
\end{align}
Since $\beta_{\mathcal{U}}(\varphi_w(z),\varphi_w(\cent)) = \beta_{\mathcal{U}}(z,\cent) \leq \ru $,
 there exists $C_{\ru} > 0$ such that the right hand side of $(\ref{5.4})$ is less than or equal to $C_{\ru}$ by Theorem \ref{MainEst}. 
Note that the constant $C_{\ru}$ is independent of $z$ and $w$. Therefore, we obtain 
$$
C_{\ru}^{-1} K_{\mathcal{U}}(\cent,\cent) \leq \n{K_{\mathcal{U}}(z,w)} \leq C_{\ru} K_{\mathcal{U}}(\cent,\cent)
$$
 for any $z \in B(\cent,\ru)$ and $w \in \mathcal{U}$. 
\end{proof}

\noindent
\textbf{Acknowledgements.} \ 
The authors would like to express our gratitude to Professors Takaaki Nomura and Takeo Ohsawa for helpful discussions.

\end{document}